\documentclass[a4paper,11pt]{amsart}
\usepackage{MyPack}
\usepackage{MyMathPack}

\newcommand{\md}[1]{\textit{#1}}
	
\begin{document}

\title{Decomposition of Symplectic Vector Bundles and Azumaya Algebras}

\author[Arcila-Maya]{Niny Arcila-Maya\\ August 23, 2025}

\date{\today}

\address{Department of Mathematics, San Francisco State University, San Francisco~CA, 94132 USA}

\email{niny.arcilamaya@sfsu.edu}

\begin{abstract}
Let $X$ be a connected CW complex. Let $\cV$ be a symplectic vector bundle of rank $2mn$ over $X$, and let $\cA$ be a topological Azumaya algebra of degree $2mn$ with a symplectic involution over $X$. We give conditions for the positive integers $m$ and $n$, and the dimension of $X$ so that $\cV$ can be decomposed as the tensor product of a symplectic vector bundle of rank $2m$ and an orthogonal vector bundle of rank $n$; and so that $\cA$ can be decomposed as the tensor product of topological Azumaya algebras of degrees $2m$ and $n$ with involutions of the first kind. To show this we compute the homomorphisms induced on homotopy groups by the tensor product $\tensor:\fSp{m} \times \fO{n} \to \fSp{mn}$.
\end{abstract}

\maketitle

\section{Introduction}

Work of Azumaya \cite{Azu1951}, followed by Auslander and Goldman \cite{AG1960}, generalized the notion of a central simple algebra over a field to Azumaya algebras defined over commutative rings. Grothendieck further extended this concept to schemes \cite{GroI1966}. In particular, Grothendieck’s definition specializes to the topological setting by taking the structure sheaf 
 $\cO$ to be the sheaf of continuous complex-valued functions. More explicitly, given a connected topological space $X$, a topological Azumaya algebra $\cA$ of degree $n$ over $X$ is a bundle of complex algebras whose fibers are isomorphic to the complex matrix algebra $\M(n,\CC)$. 

Let $k$ be a field and let $A$ be a central simple algebra over $k$. An \textit{involution} on $A$ is an additive map $\tau:A\to A$ satisfying $\tau\circ\tau = \id_{A}$ and $\tau(ab)=\tau(b)\tau(a)$ for all $a, b \in A$. Such an involution $\tau$ is said to be of the \textit{first kind} if $\tau|_{k}=\id_{k}$. Involutions of the first kind are classified as orthogonal or symplectic; see \cite{KMRT1998}.  The notion of involutions of the first kind has been generalized by Saltman, Knus, Parimala, and Srinivas to Azumaya algebras over schemes, along with the corresponding orthogonal and symplectic classifications; see \cite{SAzuInvo1978}*{Section 3}, \cite{KPS1990}, and \cite{K91book}*{III.8}. These also specialize to the topological setting. 

\begin{definition}
Let $X$ be a connected topological space, and let $\cA$ be a topological Azumaya algebra of degree $n$ over $X$. 
\begin{enumerate}
\item An \textit{involution on $\cA$}  is a morphism of fiber bundles $\tau:\cA \rightarrow \cA$ such that $\tau\circ\tau=\id_{\cA}$, and when restricted to fibers it is an involution of complex algebras. In this case, $(\cA,\tau)$ is called a \textit{topological Azumaya algebra with an involution $\tau$}.
\item The involution $\tau$ is said to be \md{symplectic} (\md{orthogonal}) if the restriction $\tau|_{\cA^{-1}(x)}:\cA^{-1}(x)\to\cA^{-1}(x)$ is  a symplectic (an orthogonal) involution of complex algebras for all $x\in X$, in the sense of the Subsection 2.1 in \cite{TAAwOI2022}.
\end{enumerate}
\end{definition}

Central simple algebras have numerous well-studied structural properties, and significant research has focused on determining the extent to which these properties generalize to Azumaya algebras defined over commutative rings and schemes. Methods from algebraic topology, particularly involving topological Azumaya algebras, have proven instrumental in shedding light on these algebraic questions; see \cites{AWtwisted2014, AW2x32014, AWpd2014, AWpdB2014, AFW15,  TAA2021, TAAwOI2022}. 

Classically, it is known that central simple algebras of degree $mn$, with $m$ and $n$ relatively prime, decompose as tensor products of central simple algebras of degrees $m$ and $n$. However, such decomposition does not generally extend to central simple algebras with involutions of the first kind \cites{ART1979, KPS1991, Merku, TAAwOI2022}. Consequently, analogous decompositions are not typically expected for Azumaya algebras with involutions defined over commutative rings, schemes, or topological spaces.
 
Antieau and Williams showed that, in general, Azumaya algebras without involution do not admit a prime decomposition. Their argument relies on proving the absence of such decompositions for topological Azumaya algebras without involution \cite{AW2x32014}*{Corollary 1.3}. On the other hand, the author provided explicit conditions under which topological Azumaya algebras without involution admit tensor product decompositions that need not be unique, and proved that such decompositions do not exist in general for arbitrary CW complexes; see \cite{TAA2021}*{Theorem 1.3, Remark 3.7}. Furthermore, the author also proved that topological Azumaya algebras of degree $mn$ with an orthogonal involution admit tensor product decompositions, albeit non-necessarily unique, when $m$ and $n$ are relatively prime integers and the base space is a CW complex of dimension at most $\min\{m, n\}$; see \cite{TAAwOI2022}*{Theorem 1.2, Theorem 1.5}. However, the existence of such positive decomposition results does not imply that tensor product decomposition occur for Azumaya algebras with involution of the first kind over commutative rings.

In this paper, we investigate the decomposition properties of topological Azumaya algebras equipped with symplectic involutions, extending earlier work from \cite{TAAwOI2022} on orthogonal involutions. In Theorem \ref{mainSp}, we provide conditions on positive integers $m, n$ and a topological space $X$ such that a topological Azumaya algebra of degree $2mn$ over $X$ with a symplectic involution decomposes as the tensor product of topological Azumaya algebras with involutions of the first kind. 
\begin{theorem}\label{mainSp}
Let $X$ be a CW complex such that $\dim(X)\leq 7$. Let $m$ and $n$ be relatively prime positive integers such that $m>1$, $n>7$, and $n$ is odd. If $\cA$ is a topological Azumaya algebra of degree $2mn$ over $X$ with a symplectic involution, then there exist topological Azumaya algebras $\cA_{2m}$ and $\cA_{n}$ of degrees $2m$ and $n$, respectively, such that $\cA_{2m}$ has a symplectic involution, $\cA_{n}$ has an orthogonal involution and is Brauer-trivial, and $\cA\iso \cA_{2m}\tensor\cA_{n}$. 
\end{theorem}

According to \cite{Steen1951}*{8.2}, there is a bijective correspondence between:
\begin{align*}\label{bcTAAwithInvo}
\begin{Bmatrix} 
\text{Isomorphism classes of degree-$n$}\\
\text{topological Azumaya algebras}\\
\text{over $X$ with an involution}\\
\text{locally isomorphic to $\tau$}
\end{Bmatrix}
\leftrightarrow
\begin{Bmatrix} 
\text{Isomorphism classes of principal}\\
\text{$\Aut\bigl(\fM{n},\tau\bigr)$-bundles}\\
\text{over $X$}
\end{Bmatrix},
\end{align*}
where $\tau$ is an involution of the first kind on $\fM{n}$. In this setting, topological Azumaya algebras with involution over $X$ are classified by $[X,\B\fPSp{n}]$ in the symplectic case, and by $[X,\B\fPO{n}]$ in the orthogonal case.

Therefore, we interpret the decomposition problem of topological Azumaya algebras with symplectic involution as a lifting problem involving classifying spaces; see diagram \eqref{cd:liftingproblem}. More precisely, to prove Theorem \ref{mainSp}, we establish in Theorem \ref{T:BSp2mn} that a map $X\rightarrow \B\fPSp{mn}$ lifts to the product $\B\fPSp{m}\times \B\fSO{n}$. This lift is realized via the map $f_{\tensor}:\B\fPSp{m} \times \B\fSO{n} \to \B\fPSp{mn}$ which is induced by the tensor product operation $\tensor:\fPSp{m} \times \fSO{n} \to \fPSp{mn}$. The existence of the lifting $\cA_{2m}\times\cA_{n}$ holds under the assumptions that $m>1$ and $n>7$ is odd, and $X$ is a CW complex of dimension $\dim(X) \leq 7$.

\begin{equation}\label{cd:liftingproblem}
\begin{tikzcd}[execute at begin picture={\useasboundingbox (-4.5,-1) rectangle (4.5,1);},row sep=large,column sep=huge]
& \B\PSp(m)\times \B\SO(n) \arrow[d,"f_{\tensor}"] \\
 X \arrow[r,"\cA"] \arrow[ur,dotted,"\cA_{2m}\times\cA_{n}",bend left=20]  & \B\PSp(mn).
\end{tikzcd}
\end{equation}

The proof of Theorem \ref{T:BSp2mn} relies critically on understanding the induced homomorphisms on homotopy groups arising from certain operations on symplectic and orthogonal groups. These operations include the $r$-fold direct sum of symplectic matrices $\splus^{r}:\fSp{m}\to\fSp{rm}$, and the tensor product operations $\tensor:\fSp{m} \times \fO{n} \to \fSp{mn}$ and $\stensor:\fSp{m} \times \fSp{n} \to \fO{4mn}$, which will be defined in detail in Sections \ref{sec:Sp}. We study these induced homomorphisms within the homotopy-theoretic \textit{stable range for $\fSp{n}$} defined as the range $\{0,1,\dots,4n+1\}$. In particular, we establish the following precise calculation:
\begin{proposition}
Let $i<4m+2$ and $i<n-1$, and let $\tensor:\fSp{m} \times \fO{n} \to \fSp{mn}$ denote the tensor product. This operation induces a homomorphism on homotopy groups $\tensor_{*}:\pi_{i}\fSp{m}\times\pi_{i}\fO{n} \to \pi_{i}\fSp{mn}$ given by 
\[
\otimes_{*}(x,y)=nx+2my \quad \text{for} \quad x \in \pi_{i}\fSp{m} \quad \text{and} \quad y\in \pi_{i}\fO{n}.
\]
\end{proposition}

When restricting attention to topological Azumaya algebras arising from symplectic vector bundles--those of the form $\End(\cV)$, with $\cV$ a symplectic vector bundle--we obtain a stronger decomposition result applicable in a broader dimensional range. This enhanced result is achieved by analyzing the Moore-Postnikov tower associated with the map $f_{\tensor}$.
\begin{theorem}\label{vbSp}
Let $m$ and $n$ be positive integers such that $n$ is odd. Let $X$ be a CW complex such that $\dim(X)\leq n$. If $\cV$ is a symplectic bundle of rank $2mn$ over $X$, then there exist a symplectic bundle $\cV_{2m}$ of rank $2m$, and an orthogonal bundle $\cV_{n}$ of rank $n$ such that $\cV\iso \cV_{2m}\tensor\cV_{n}$.
\end{theorem}

This paper is organized as follows. In Section \ref{sec:Sp} we introduce the direct sum and tensor product operations for compact Lie groups related to the complex symplectic and complex orthogonal groups. We review the homotopy groups of the complex symplectic and complex projective symplectic groups. Moreover, we establish the connectivity of the stabilization maps. Section \ref{sec:stabilization} presents results describing the homomorphisms induced on homotopy groups by the operations defined in Preliminaries. Section \ref{sec:mainproof} is devoted to the proofs of Theorem \ref{mainSp} and Proposition \ref{nosecsymp}. We prove in Proposition \ref{nosecsymp} that $f_{\tensor}$ does not have a section necessarily and we provide examples of topological Azumaya algebras with symplectic involutions that do not decompose as the tensor product of topological Azumaya algebras with involutions of the first kind when defined over spaces whose dimension is outside the stable range for symplectic groups. Finally, Section \ref{sec:symplecticvb} we prove the decomposition of complex symplectic vector bundles in the stable range for the complex special orthogonal group.

\subsection{Acknowledgments}
The author extends gratitude to \textit{Ben Williams} for suggesting the research topic,  investing substantial time in detailed research discussions, and for thoroughly reviewing the manuscript. Heartfelt appreciation is also extended to \textit{Kirsten Wickelgren} for valuable suggestions that significantly improved the manuscript’s exposition. The author is grateful to the anonymous referee for their careful reading and insightful comments, many of which have been incorporated into this revised version.

\subsection{Notation and Conventions} \label{notation}
All topological spaces considered are assumed to have the homotopy type of a CW complex. We fix basepoints for connected topological spaces, and for topological groups we take the identity element as the basepoint. We write $\pi_{i}X$ in place of $\pi_{i}(X,x_{0})$. 

Let $\FF$ denote either the field of real numbers $\RR$ or complex numbers $\CC$. We denote by $\M(n,\FF)$ the group of $n\times n$ matrices with entries in $\FF$, and by $\GL(n,\FF)\coloneqq \M(n,\FF)^{\times}$ the group of invertible matrices. 
We write $\SL(n,\FF)$, $\Or(n,\FF)$, and $\SO(n,\FF)$ for the special linear group, the orthogonal group, and the special orthogonal group of degree $n$ over $\FF$, respectively. The complex unitary group of degree $n$ is denoted by $\U(n,\CC)$, the compact symplectic group of degree $n$ by $\U(n,\HH)$, and the complex symplectic group of degree $2n$ by $\fSp{n}$. 
The complex projective special group of degree $n$ is denoted by $\fPSL{n}\coloneqq \fSL{n}/Z(\fSL{n})$, and the complex projective sympletic group of degree $2n$ by $\fPSp{n}\coloneqq \fSp{n}/Z(\fSp{n})$, where $Z(-)$ denotes the center of a group.

Given matrices $A_{1}, \dots, A_{r} \in \M(n,\FF)$, we write $\diag(A_{1},\dots,A_{r}) \in \M(rn,\FF)$ to denote the block-diagonal matrix with blocks $A_{1}, \dots, A_{r}$ along the diagonal.

\section{Preliminaries}
\label{sec:Sp}

In this section, we introduce operations and notation that will be central to the solution of the lifting problem. We begin by introducing a set of operations and maps on the symplectic and orthogonal groups. We then review the homotopy groups of $\fSp{n}$, and compute the corresponding groups for $\fPSp{n}$ in low degrees. Finally, we establish the connectivity of the stabilization map and its componentwise version, which justifies a simplified notation that will streamline the discussion of stabilized operations in Sections \ref{sec:stabilization} and \ref{sec:mainproof}.

\subsection{Operations on \texorpdfstring{$\Sp(n,\CC)$}{Sp(n,CC)}} \label{subsub:operations}
We introduce the direct sum of symplectic matrices, the stabilization maps of symplectic groups, the doubling map, and two tensor product operations. These will be used to construct the map $\widetilde{\Tr}:\fPSp{m}\times \fSO{n} \rightarrow \fSO{N}$ in Proposition \ref{P:TtildaSp} and to understand the tensoring map $f_{\tensor}:\B\fPSp{m} \times \B\fSO{n} \to \B\fPSp{mn}$ at the level of homotopy groups. These maps serve as key building blocks for the lifting $\cA_{2m}\times \cA_{n}$ in diagram \eqref{cd:liftingproblem}. 

Let $A\in \fSp{n}$. Then we may write $A$ in block form as
\[
\begin{pmatrix}
A_{11} & A_{12}\\
A_{21} & A_{22}
\end{pmatrix}
\]
where each $A_{ij} \in \M(n,\CC)$ for $i,j=1,2$, and the following conditions are satisfied: $A_{11}^{\tr}A_{21}$ and $A_{12}^{\tr}A_{22}$ are symmetric, and $A_{11}^{\tr}A_{22}-A_{21}^{\tr}A_{12}=I_{n}$.

\begin{definition}
Let $m, n, r \in \NN$. We define the following maps and operations:
\begin{enumerate}
\item The \textit{standard inclusion of the symplectic group} is the map $i:\fSp{n} \hookrightarrow \fSp{n+1}$ defined by
\begin{equation*}
i(A)=
\begin{pmatrix}
 &  &  & 0 &  &  &  & 0\\
 & A_{11} &  & \vdots &  & A_{12} &  & \vdots \\
 &  &  & 0 &  &  &  & 0\\
0 & \cdots & 0 & 1 & 0 & \cdots & 0 & 0 \\
 &  &  & 0 &  &  &  & 0\\
 & A_{21} &  & \vdots &  & A_{22} &  & \vdots \\
 &  &  & 0 &  &  &  & 0\\
0 & \cdots & 0  & 0 &  0 & \cdots & 0  & 1
\end{pmatrix}.
\end{equation*}
\item The \textit{direct sum of symplectic matrices} is the operation $\splus:\fSp{m}\times \fSp{n} \to \fSp{m+n}$ defined by 
\[
A\splus B=
\left(\begin{NiceArray}{cc:cc}
A_{11} &  & A_{12} & \\
 & B_{11} &  & B_{12}\\[1mm]
\hdottedline
A_{21} &  & A_{22} & \\
 & B_{21} &  & B_{22}
\end{NiceArray}\right).
\]

\item The \textit{$r$-fold direct sum of symplectic matrices} is the operation $\splus^{r}:\fSp{n}\to \fSp{rn}$ defined by 
\[
A^{\splus r}=
\left(\begin{NiceArray}{cc}
A_{11}^{\oplus r} & A_{12}^{\oplus r} \\[2mm]
A_{21}^{\oplus r} & A_{22}^{\oplus r}
\end{NiceArray}\right).
\]

\item The \textit{stabilization map} $\s:\fSp{m} \rightarrow \fSp{m+n}$ is defined by
\begin{equation*}
\s(A)=
\left(\begin{NiceArray}{cc:cc}
A_{11} &  & A_{12} & \\
 & I_{n} &  & I_{n} \\[1mm]
\hdottedline
A_{21} &  & A_{22} & \\
 & I_{n}  &  & I_{n} 
\end{NiceArray}\right).
\end{equation*}

\item For each $j=1,\dots,r$, the \textit{$j$-th stabilization map} $\s_{j}:\fSp{n} \rightarrow \fSp{rn}$ is define by 
\[
\s_{j}(A)=
\begin{pmatrix}
\sigma_{j}(A_{11}) & \sigma_{j}(A_{12})\\[1.5mm]
\sigma_{j}(A_{21}) & \sigma_{j}(A_{22})
\end{pmatrix},
\]
where 
\[
\sigma_{j}(X)=\diag(I_{n},\dots,I_{n},X,I_{n},\dots,I_{n}) \in \M(rn,\CC),
\]
with $X \in \M(rn,\CC)$ appearing in the $j$-th diagonal block and the remaining blocks being $n\times n$ identity matrices. 
\end{enumerate}
\end{definition}

\begin{definition}
Consider the two-fold direct sum on orthogonal groups:
\[
\oplus^{2}(A)=%
\begin{pmatrix}
A & 0\\
0 & A
\end{pmatrix}.
\]
Since the image of $\oplus^{2}$ is contained in $\fSp{n}$, we define the \textit{doubling map} as the homomorphism $d:\fO{n}\to \fSp{n}$, $d(A)=\oplus^{2}(A)$.
\end{definition}

Let us denote by $B_{\mathrm{sym}}$ the standard symmetric bilinear form and by $B_{\mathrm{skew}}$ the standard skew-symmetric bilinear form. Recall that with this notation we have the isomorphisms
\[
\Aut(\CC^{n}, B_{\mathrm{sym}}) \iso \fO{n} \quad \text{ and } \quad \Aut(\CC^{2n}, B_{\mathrm{skew}}) \iso \fSp{n}.
\]

\begin{definition}
Let $m, n \in \NN$. We defined two distinct tensor product operations. 
\begin{enumerate}
\item The \textit{tensor product of symplectic and orthogonal matrices} is the operation 
\begin{equation*}
\begin{tikzcd}
\tensor:\fSp{m} \times \fO{n}\arrow[r] & \fSp{mn}
\end{tikzcd}
\end{equation*}
induced by the tensor product of bilinear forms $\tensor: \Aut(\CC^{2m}, B_{\mathrm{skew}}) \times \Aut(\CC^{n}, B_{\mathrm{sym}})\rightarrow \Aut(\CC^{2m}\tensor \CC^{n}, B_{\mathrm{skew}}\tensor B_{\mathrm{sym}})$.

\item The tensor product of symplectic matrices arises from the tensor product of bilinear forms $\tensor: \Aut(\CC^{2m}, B_{\mathrm{skew}}) \times (\CC^{2m}, B_{\mathrm{skew}})\rightarrow G$ where $G\leq \fGL{4mn}$ denotes $\Aut(\CC^{4mn}, B_{\mathrm{skew}} \tensor B_{\mathrm{skew}})$. Note that $B_{\mathrm{skew}}\tensor B_{\mathrm{skew}}$ is not the standard symmetric bilinear form on $\CC^{4mn}$. Let $P \in \fGL{4mn}$  be the matrix representing a change to an orthonormal basis. We then have a commutative diagram, where $\Inn_{P}$ denotes conjugation by $P$:
\begin{equation*}
\begin{tikzcd}[row sep=large]
G \arrow[r,hookrightarrow] \arrow[d,"\Inn_{P}"] & \GL(4mn,\CC) \arrow[d,"\Inn_{P}"]\\
\Or(4mn,\CC) \arrow[r,hookrightarrow] \arrow[u,leftarrow,"\iso"] & \GL(4mn,\CC). \arrow[u,leftarrow,"\iso"]
\end{tikzcd}
\end{equation*}

The \textit{tensor product of symplectic matrices} is thus defined as the composite:
\begin{equation*}
\begin{tikzcd}
\stensor:\fSp{m} \times \fSp{n}\arrow[r] &  \fO{4mn}, \quad \stensor\coloneq \Inn_{P} \circ \tensor.
\end{tikzcd}
\end{equation*}
\end{enumerate}
\end{definition}

\subsection{Homotopy groups of \texorpdfstring{$\fSp{n}$}{Sp(n)} and \texorpdfstring{$\fPSp{n}$}{PSp(n)}} \label{subsub:homotype}
We recall the homotopy groups of $\fSp{n}$ in the stable range, and its first unstable homotopy group. We compute the homotopy groups of $\fPSp{n}$ in low degrees.

Let $n\geq1$. The standard inclusion defined in the previous subsection $i:\fSp{n}\hookrightarrow \fSp{n+1}$ is $(4n+2)$-connected. Thus, it induces an isomorphism on homotopy groups in degrees less than $4n+2$ and an epimorphism in degree $4n+2$. This connectivity follows from examining the long exact homotopy sequence associated with the fibration $\fSp{n} \hookrightarrow \fSp{n+1} \rightarrow \fSp{n+1}/\fSp{n}\simeq S^{4n+3}$. According to \cite{MiToToG1991}*{Chapter IV}, the homotopy groups of $\fSp{n}$ in the stable range are computed using Bott periodicity. They are given by:
\begin{align*}
\pi_{i}\fSp{n}&\iso
\begin{cases}
0 & \text{if $i=0,1,2,6$ (mod 8),}\\
\ZZ/2 & \text{if $i=4,5$ (mod 8),}\\
\ZZ & \text{if $i=3,7$ (mod 8),}
\end{cases}
\end{align*}
 for all $i< 4n$, and by
\begin{align*}
\pi_{4n}\fSp{n}\iso\pi_{4n+1}\fSp{n}&\iso
\begin{cases}
\ZZ/2 & \text{if $n$ is odd},\\
0 & \text{if $n$ is even},
\end{cases}
\end{align*}
at the boundary of the stable range. The first unstable homotopy group of $\fSp{n}$ occurs in degree $4n+2$ and, from  \cite{MiToHoG1963}, it is explicitly given by:
\begin{align*}
\pi_{4n+2}\fSp{n}\iso
\begin{cases}
\ZZ/(2n+1)! & \text{if $n$ is even,}\\
\ZZ/(2n+1)!\cdot2 & \text{if $n$ is odd}.
\end{cases}
\end{align*}

The projective symplectic group $\fPSp{n}$ relates closely to $\fSp{n}$. 
The quotient map $\fSp{n} \rightarrow \fPSp{n}$ is a universal cover with fiber $\ZZ/2$, yielding immediately that $\pi_{1}\fPSp{n}\iso \ZZ/2$. Moreover, using the associated fibration $\{\pm I_{2n}\}\hookrightarrow \fSp{n} \rightarrow \fPSp{n}$ we deduce from the long exact sequence of homotopy groups that 
\[
\pi_{i}\fPSp{n}\iso\pi_{i}\fSp{n} \quad \text{for all } i\ge 2.
\]

\begin{table}[!ht]
\centering
{\small \renewcommand{\arraystretch}{1.2}
\begin{tabular}{|c|cccccccc|}
\hline
$i>1$ and $i$ (mod 8) & 0 & 1 & 2 & 3 & 4 & 5 & 6 & 7\\
\hline
$\pi_{i}G$ & 0 & 0 & 0 & $\ZZ$ & $\ZZ/2$ & $\ZZ/2$ & 0 & $\ZZ$ \\
\hline
\end{tabular}}
\caption{Homotopy groups of $\fSp{n}$ and $\fPSp{n}$ for $i=2,\dots,4n-1$}
\label{tab:lowhgSp}
\end{table}

\subsection{Stabilization notation} \label{subsub:stab}
We prove that the stabilization map and the componentwise stabilization maps are connected in the stable range for $\fSp{n}$. This result allows us to introduce a convenient notation for stabilizing the operations and homomorphisms defined in Subsection \ref{subsub:operations}.

Let $m, n \in \NN$. We begin by noting that the stabilization map $\s:\fSp{m} \rightarrow \fSp{m+n}$ can be expressed as the composite of consecutive canonical inclusions. Consequently, the map $\s$ is $(4m+2)$-connected. Let $\esta$ denote the homomorphism induced by $\s$ on homotopy groups. From this connectivity, we identify $\pi_{i}\fSp{m}$ with $\pi_{i}\fSp{m+n}$ for all $i<4m+2$ via the isomorphism induced by $\s$:
\begin{equation}
\begin{tikzcd}
\esta:\pi_{i}\fSp{m} \arrow[r,"\iso"] & \pi_{i}\fSp{m+n}.
\end{tikzcd}
\end{equation}

\begin{lemma}\label{L:sjSp}
Let $n,r \in \NN$. The maps $\s_{1},\dots, \s_{r}:\fSp{n} \rightarrow \fSp{rn}$ are pointed homotopic.
\end{lemma}
\begin{proof}
We show that the maps $\s_{j}$ and $\s_{j+1}$ are pointed homotopic for $j=1,\dots,r-1$. 

For each such $j$, define the permutation matrix $P_{j} \in \M(rn,\CC)$ by
\[
P_{j}=
\begin{pmatrix}
I_{(j-1)n} & & &  \\
 & 0 & I_{n} & \\
 & I_{n} & 0 & \\
 & & & I_{(r-j-1)n}
\end{pmatrix}.
\]
This matrix acts by transposing the $j$-th and $(j+1)$-st $n\times n$ blocks along the diagonal. Observe that $P_{j}$ satisfies the relation
\[
\sigma_{j+1}(X)=P_{j}\sigma_{j}(X)P_{j} \quad \text{for all } X \in \M(n,\CC),  
\]
which can be verified by direct matrix multiplication. For instance, when $j=2$, $P_{2}$ and $\sigma_{2}(X)$ take the form: 
\begin{align*}
P_{2}=
\begin{pmatrix}
I_{n} & & &  \\
 & 0 & I_{n} & \\
 & I_{n} & 0 & \\
 & & & I_{(r-3)n}
\end{pmatrix} 
\quad \text{and} \quad
\sigma_{2}(X)=
\begin{pmatrix}
I_{n} & & &  \\
 & X & 0 & \\
 & 0 & I_{n} & \\
 & & & I_{(r-3)n}
\end{pmatrix},
\end{align*}
and one checks that 
\begin{align*}
P_{2}\sigma_{2}(X)P_{2}
=
\begin{pmatrix}
I_{n} & & &  \\
 & I_{n} & 0 & \\
 & 0 & X & \\
 & & & I_{(r-3)n}
\end{pmatrix}
= \sigma_{3}(X).
\end{align*}

As a consequence, for all $A \in \fSp{n}$  we have
\[
\s_{j+1}(A)=\diag(P_{j},P_{j})\s_{j}(A)\diag(P_{j},P_{j}).
\]
To seee this, write 
\[
A=
\begin{pmatrix}
A_{11} & A_{12}\\
A_{21} & A_{22}
\end{pmatrix},
\]
and compute:
\begin{align*}
\diag(P_{j},P_{j})\s_{j}(A)\diag(P_{j},P_{j}) &= 
\begin{pmatrix}
P_{j} & 0\\
0 & P_{j}
\end{pmatrix}
\begin{pmatrix}
\sigma_{j}(A_{11}) & \sigma_{j}(A_{12})\\[1.5mm]
\sigma_{j}(A_{21}) & \sigma_{j}(A_{22})
\end{pmatrix}
\begin{pmatrix}
P_{j} & 0\\
0 & P_{j}
\end{pmatrix}\\
&=%
\begin{pmatrix}
\sigma_{j+1}(A_{11}) & \sigma_{j+1}(A_{12})\\[1.5mm]
\sigma_{j+1}(A_{21}) & \sigma_{j+1}(A_{22})
\end{pmatrix}
= \s_{j+1}(A).
\end{align*}

Thus, $s_{j+1}=c\circ s_{j}$, where $c:\fSp{rn} \rightarrow \fSp{rn}$ is conjugation by $\diag(P_{j},P_{j})$; that is, $c(X)=\diag(P_{j},P_{j})X\diag(P_{j},P_{j})^{-1}$ for all $X\in \fSp{rn}$. Since $\diag(P_{j},P_{j}) \in \fSp{rn}$ and $\fSp{rn}$ is path-connected, it follows from \cite{TAAwOI2022}*{Lemma 3.2} that $c$ is pointed homotopic to the identity map on $\fSp{rn}$. Therefore, $\s_{j}$ and $\s_{j+1}$ are pointed homotopic for $j=1,\dots,r-1$.
\end{proof}

Since the first stabilization map $\s_{1}$ is equal to $\s:\fSp{n} \to \fSp{n+(r-1)n}$, it follows that each stabilization map $\s_{j}$ is $(4n+2)$-connected for all $j=1,\dots,r$. By Lemma \ref{L:sjSp}, the induced homomorphisms on homotopy groups by the componentwise stabilization maps are equal. Hence, we use $\esta$ to denote the common map
\[
\esta = \pi_{i}(\s_{1}) = \cdots = \pi_{i}(\s_{r}). 
\]
Thus, we identify $\pi_{i}\fSp{n}$ with $\pi_{i}\fSp{rn}$ for $i<4n+2$ via $\esta$. This identification justifies a slight abuse of notation: we write $x = \esta(x)$ for $x\in \pi_{i}\fSp{n}$ and $i<4n+2$.

\section{Stabilization of operations on homotopy groups of \texorpdfstring{$\Sp(n,\CC)$}{Sp(n,CC)} and \texorpdfstring{$\PSp(n,\CC)$}{PSp(n,CC)}} \label{sec:stabilization}
In this section, we describe the homomorphisms induced on homotopy groups by the operations and maps introduced in Section \ref{sec:Sp}. These descriptions play a crucial role in establishing the connectivity of the map $J:\B\PSp(m)\times\B\SO(n) \to \B\PSp(mn)\times\B\SO(N)$ as detailed in Proposition \ref{P:JOSpi}. Of particular significance are the homomorphisms induced by the tensor product operations. Specifically, we show that for $n$ an odd integer, the tensor product operation involving symplectic and orthogonal matrices descends naturally to a quotient $\tensor:\fPSp{m}\times \fSO{n} \to \fPSp{m}$, and we explicitly describe the induced homomorphism within the stable range for $\fSp{m}$. Throughout, the subscript star notation (e.g. $\splus_{*}, \tensor_{*}, \stensor_{*}$, etc) will denote the induced maps on homotopy groups.

We do not write proofs for the following statements as they are analogous to the proofs in \cite{TAA2021}.
\begin{proposition}\label{P:dsSp}
\cite{TAA2021}*{Proposition 2.5}
Let $i\in \NN$. The homomorphism $\splus_{*}: \pi_{i}\fSp{m}\times \pi_{i}\fSp{n} \rightarrow \pi_{i}\fSp{m+n}$ is given by
\[
\splus_{*}(x,y)=\esta(x)+\esta(y) \quad \text{for} \quad x\in \pi_{i}\fSp{m} \quad \text{and} \quad y\in\pi_{i}\fSp{n}.
\]
\end{proposition}

\begin{corollary}\label{C:dsSp}
\cite{TAA2021}*{Corollary 2.6}
If $m<n$ and $i<4m+2$, then 
\[
\splus_{*}(x,y)=x+y \quad \text{for} \quad x\in \pi_{i}\fSp{m} \quad \text{and} \quad y\in\pi_{i}\fSp{n}.
\]
\end{corollary}

\begin{proposition}\label{P:rdsSp}
\cite{TAA2021}*{Proposition 2.7}
Let $i\in \NN$. The homomorphism $\splus^{r}_{*}:\pi_{i}\fSp{n}\rightarrow \pi_{i}\fSp{rn}$ is  given by
\[
\splus^{r}_{*}(x)=r\esta(x)  \quad \text{for} \quad x\in \pi_{i}\fSp{n}.
\]
\end{proposition}

\begin{corollary}\label{C:rdsSp}
\cite{TAA2021}*{Corollary 2.8}
If $i<4n+2$, then 
\[
\splus_{*}^{r}(x)=rx \quad \text{for} \quad x\in \pi_{i}\fSp{n}.
\]
\end{corollary}

The following maps will be used in the proofs presented in the subsequent subsections.
\begin{definition}
Let $p, q \in \NN$. We denote by $L$ and $R$ the restrictions of the tensor product map $\tensor:\fM{p}\times\fM{q}\to \fM{pq}$ to $\fM{p} \times \{I_{q}\}$ and $\{I_{p}\} \times  \fM{q}$, respectively. This is, $L$ and $R$ are the homomorphisms:
\[
L:\fM{p}\rightarrow \fM{pq}, \quad L(A)=A\tensor I_{q},
\]
and
\[
R:\fM{n}\rightarrow \fM{pq}, \quad R(B)=I_{p}\tensor B,
\]
for all $A \in \fM{p}$ and $B \in \fM{q}$.
\end{definition}
By the mixed-product property of tensor products, we observe that for all $A \in \fM{p}$ and $B \in \fM{q}$:
\begin{equation}\label{mixedproduct}
A\tensor B=(A\tensor I_{q})(I_{p}\tensor B)=L(A)R(B)
\end{equation}
for all $A \in \fM{p}$ and $B \in \fM{q}$.

\subsection{Tensor product of symplectic and orthogonal matrices in the stable range}
This subsection aims to prove Proposition \ref{P:tpSpO}, which determines the homomorphisms induced on homotopy groups by the tensor product map on symplectic and orthogonal groups $\tensor:\fSp{m} \times \fO{n}\to \fSp{mn}$. 


\begin{lemma}\label{L:tpLSpO}
There is a basepoint preserving homotopy $H$ from $L$ to the $n$-fold direct sum map of symplectic matrices $\splus^{n}:\fSp{m} \to \fSp{mn}$ such that for all $t \in [0,1]$, $H(-,t)$ is a homomorphism.
\end{lemma}
\begin{proof}
Let $A\in \fSp{m}$,
\[
L(A)=%
\begin{pmatrix}
A_{11}\tensor I_{n} & A_{12}\tensor I_{n}\\
A_{21}\tensor I_{n} & A_{22}\tensor I_{n}\\
\end{pmatrix}
\quad \text{and} \quad %
A^{\splus n}=
\begin{pmatrix}
A_{11}^{\oplus n} & A_{12}^{\oplus n}\\[1.5mm]
A_{21}^{\oplus n} & A_{22}^{\oplus n}
\end{pmatrix}.
\]

Let $P_{m,n}$ be the permutation matrix
\begin{align*}
P_{m,n}=[&e_{1},e_{n+1},e_{2n+1},\dots, e_{(m-1)n+1},%
		      e_{2},e_{n+2},e_{2n+2},\dots, e_{(m-1)n+2},\\%
		   &\dots,\\
		   &e_{n-1},e_{2n-1},e_{3n-1},\dots,e_{mn-1},%
		     e_{n},e_{2n},e_{3n},\dots, e_{(m-1)n},e_{mn}]
\end{align*}
where $e_{i}$ is the $i$-th standard basis vector of $\CC^{mn}$ written as a column vector. Observe that $A_{ij}\tensor I_{n}=P_{m,n}\,A_{ij}^{\oplus n}\,P_{m,n}^{-1}$ for $i,j=1,2$. This implies
\[
L(A)=%
\begin{pmatrix}
P_{m,n} & 0\\
0 & P_{m,n}
\end{pmatrix}%
A^{\splus n}%
\begin{pmatrix}
P_{m,n}^{-1} & 0\\
0 & P_{m,n}^{-1}
\end{pmatrix}.
\]

The result follows from \cite[Lema 3.2]{TAAwOI2022}.
\end{proof}


\begin{lemma}\label{L:tpRSpO}
Let $i\in \NN$. Then
\[
R_{*}(x)=m\esta(d_{*}(x)) \quad \text{for} \quad x \in \pi_{i}\fO{n},
\]
where $d_{*}:\pi_{i}\fO{n}\to \pi_{i}\fSp{mn}$ is the homomorphism induced by the doubling homomorphism.
\end{lemma}
\begin{proof}
Consider the maps $\sigma_{j}:\fM{2n}\to \fM{2mn}$ given by 
\[
\sigma_{j}(X)=\diag(I_{2n},\dots,I_{2n},X,I_{2n},\dots,I_{2n})
\]
where $X$ is the $j$-th position for all $j=1,\dots,m$. Observe that $R$ can be decomposed as the composite:
\begin{equation*}
\begin{tikzcd}
\fO{n} \arrow[d,"d",swap] & A \arrow[d,mapsto] \\
\fSp{n} \arrow[d,"\Delta",swap] & d(A) \arrow[d,mapsto] \\
(\fSp{n})^{\times m} \arrow[d,"\sigma_{1}\times \cdots \times\sigma_{m}",swap] & \Bigl(d(A),\dots,d(A)\Bigr) \arrow[d,mapsto] \\
(\fSp{mn})^{\times m} \arrow[d,"\m",swap]  & \Bigl((\sigma_{1}\circ d)(A),\dots,(\sigma_{m}\circ d)(A)\Bigr) \arrow[d,mapsto] \\
\fSp{mn} & (\sigma_{1}\circ d)(A)\boldsymbol{\cdot}\cdots\boldsymbol{\cdot}(\sigma_{m}\circ d)(A).
\end{tikzcd}
\end{equation*}
By \cite[Lemma 3.3, Proposition 3.7]{TAAwOI2022}, all maps $\sigma_{j}$ are homotopic to $\sigma_{1}$, and thus, $R_{*}(x)=m\esta(d_{*}(x))$.
\end{proof}

\begin{proposition}\label{tensorSpO}
Let $i\in\NN$. The homomorphism $\tensor_{*}:\pi_{i}\fSp{m}\times\pi_{i}\fO{n}\to\pi_{i}\fSp{mn}$ is given by
\[
\tensor_{*}(x,y)=n\esta(x) + m\esta(d_{*}(y)) \quad \text{for} \quad x\in \pi_{i}\fSp{m} \quad \text{and} \quad y \in \pi_{i}\fO{n}.
\]
\end{proposition}
\begin{proof}
Let $x\in \pi_{i}\fSp{m}$ and $y \in \pi_{i}\fO{n}$. By Proposition \ref{P:rdsSp}, Lemmas \ref{L:tpLSpO} and \ref{L:tpRSpO}, and equation \eqref{mixedproduct} we have that:
\[
\tensor_{*}(x,y)=L_{*}(x)+R_{*}(y)=\splus^{n}_{*}(x)+\oplus^{m}_{*}(d_{*}(y))=n\esta(x)+m\esta(d_{*}(y))
\]
for all $i\in \NN$.
\end{proof}

In view of Proposition \ref{tensorSpO}, we focus our attention on describing the homomorphism induced on homotopy groups by the doubling homomorphism $d$ in the stable range for $\fO{n}$. 

\begin{proposition}\label{P:doubling}
Let $i<n-1$. The homomorphism $d_{*}:\pi_{i}\fO{n} \longrightarrow \pi_{i}\fSp{n}$ 
is an isomorphism onto $2\pi_{i}\fSp{n}$ if $i\equiv 3,7$ (mod 8). In all other cases, either the source or the target is trivial.
\end{proposition}
\begin{proof}
Let $c:\Or(n,\RR)\hookrightarrow \fU{n}$ and $q:\fGL{n} \to \GL(n,\HH)$ be the inclusions. Since the conjugation on $\HH$ extends that on $\CC$, the inclusion $q$ restricts to $q:\fU{n} \to \U(n,\HH)$. 

The homomorphism of $\RR$-algebras $g: \M(n,\HH) \to  \fM{2n}$ given by
\[
g(A + jB)=%
\begin{pmatrix}
A & -\overline{B}\\
B & \overline{A}
\end{pmatrix}
\]
is injective, and its restriction to $\U(n,\HH)$ is an isomorphism, $g: \U(n,\HH) \longrightarrow \fU{2n}\cap \fSp{n}$  \cite[pp. 22--23]{MiToToG1991}. Thus we arrive at the following homomorphism
\begin{equation*}
\begin{tikzcd}[column sep=large, row sep=0.5pt]
\fU{n} \arrow[r,"q"] & \U(n,\HH) \arrow[r,"g"] & \fU{2n}\cap\fSp{n} \arrow[r,hook,"\simeq"] \arrow[l,leftarrow,"\iso"]& \fSp{n}\\
A \arrow[rrr,mapsto] & & & \begin{pmatrix}
A & 0\\
0 & \overline{A}
\end{pmatrix}.
\end{tikzcd}
\end{equation*}

Hence, we can factor $d$ as follows
\begin{equation*}
\begin{tikzcd}[row sep=large]
\fO{n} \arrow[rr,"d"]  &  & \fSp{n} \\
\Or(n,\RR) \arrow[r,"c"] \arrow[u,hook,"\simeq"] & \fU{n} \arrow[r,"q\circ g"] & \fU{2n}\cap\fSp{n}. \arrow[u,hook,"\simeq"]
\end{tikzcd}
\end{equation*}

The effect of $d_{*}$ on homotopy groups in the stable range for $\fO{n}$ is identified with the effect of $q_{*}\circ g_{*}\circ c_{*}$. 
From \cite[I.4.11, I.4.12]{MiToToG1991}, \cite[Theorem IV.5.12]{MiToToG1991}, and \cite[IV.6.1.(2)]{MiToToG1991} the result follows.
\end{proof}

\begin{proposition}\label{P:tpSpO}
Let $i<4m+2$ and $i<n-1$, the homomorphism $\tensor_{*}:\pi_{i}\fSp{m}\times\pi_{i}\fO{n} \to \pi_{i}\fSp{mn}$
is given by 
\[
\otimes_{*}(x,y)=nx+2my \quad \text{for} \quad x \in \pi_{i}\fSp{m} \quad \text{and} \quad y\in \pi_{i}\fO{n}.
\]
\end{proposition}
\begin{proof}
Corollary \ref{C:rdsSp}, \cite[Corollary 3.8]{TAAwOI2022}, and Propositions \ref{tensorSpO} and \ref{P:doubling} yield the result.
\end{proof}

\subsection{Tensor product of symplectic matrices in the stable range}
This subsection aims to prove Corollary \ref{C:tpSpSp}, which determine the homomorphisms induced on homotopy groups by the tensor product map on symplectic groups $\stensor:\fSp{m} \times \fSp{n}\to \fO{4mn}$. Recall that $\stensor:\Sp(m) \times \fSp{n} \to \Or(4mn)$ is equal to the composite 
\begin{equation*}
\begin{tikzcd}
\fSp{m} \times \fSp{n} \arrow[r,"\tensor"] & G \arrow[r,"\Inn_{P}"] & \fO{4mn},
\end{tikzcd}
\end{equation*}
where $G\leq \fGL{4mn}$ and $P\in \fGL{4mn}$ is the matrix representing a change to an orthonormal matrix as explained in Section \ref{subsub:operations}. 

\begin{proposition}\label{P:tpSpSp}
Let $m\leq n$ and $i<4m+2$, the homomorphism 
\[
\bigl(c''_{*}\circ\stensor_{*}\bigr):\pi_{i}\fSp{m}\times\pi_{i}\fSp{n} \to \pi_{i}\fGL{4mn}
\]
is given by
\[
\bigl(c''_{*}\circ\stensor_{*}\bigr)(x,y)=2nc'_{*}(x)+2mc'_{*}(y)  \quad \text{for} \quad x \in \pi_{i}\fSp{m} \quad \text{and} \quad y \in \pi_{i}\fSp{n},
\]
where $c'$ and $c''$ denote the inclusions
\begin{equation*}
\begin{tikzcd}
c':\fSp{n} \arrow[r,hook] & \fGL{2n} &\text{and}& c'':\fO{n} \arrow[r,hook] & \fGL{2n}.
\end{tikzcd}
\end{equation*}
\end{proposition}
\begin{proof}
Consider the commutative diagrams

\begin{minipage}{0.48\linewidth}
\begin{equation}\label{cd:cLL}
\begin{tikzcd}[row sep=large]
\fSp{m} \arrow[r,"c'"] \arrow[d,"L"] &  \fGL{2m} \arrow[d,"L"] \\
G \arrow[r,hookrightarrow] \arrow[d,"\Int_{P}"] &  \fGL{4mn} \arrow[d,"\Int_{P}"] \\
\fO{4mn} \arrow[r,"c''"] \arrow[u,leftarrow,"\iso"] & \fGL{4mn} \arrow[u,leftarrow,"\iso"]
\end{tikzcd}
\end{equation}
\end{minipage}
\begin{minipage}{0.48\linewidth}
\begin{equation}\label{cd:cRR}
\begin{tikzcd}[row sep=large]
\fSp{n} \arrow[r,"c'"] \arrow[d,"R"] &  \fGL{2n} \arrow[d,"R"] \\
G \arrow[r,hookrightarrow] \arrow[d,"\Int_{P}"] &  \fGL{4mn} \arrow[d,"\Int_{P}"] \\
\fO{4mn} \arrow[r,"c''"] \arrow[u,leftarrow,"\iso"] & \fGL{4mn} \arrow[u,leftarrow,"\iso"]
\end{tikzcd}
\end{equation}
\end{minipage}

\vspace{5mm}

From these we obtain the commutative squares below:

\begin{minipage}{0.48\linewidth}
\begin{equation}\label{cd:cL}
\begin{tikzcd}[row sep=large]
\pi_{i}\fSp{m} \arrow[r,"c'_{*}"] \arrow[d,"L_{*}"] &  \pi_{i}\fGL{2m} \arrow[d,"L_{*}"]\\
\pi_{i}\fO{4mn} \arrow[r,"c''_{*}"] & \pi_{i}\fGL{4mn}
\end{tikzcd}
\end{equation}
\end{minipage}
\begin{minipage}{0.48\linewidth}
\begin{equation}\label{cd:cR}
\begin{tikzcd}[row sep=large]
\pi_{i}\fSp{n} \arrow[r,"c'_{*}"] \arrow[d,"R_{*}"] &  \pi_{i}\fGL{2n} \arrow[d,"R_{*}"]\\
\pi_{i}\fO{4mn} \arrow[r,"c''_{*}"] & \pi_{i}\fGL{4mn}
\end{tikzcd}
\end{equation}
\end{minipage}

\vspace{5mm}

Observe that the map $R:\fGL{2n}\to \fGL{4mn}$ in diagram \eqref{cd:cRR} is equal to the $(2n)$-fold direct sum. Then for all $i\in \NN$, by \cite{TAA2021}*{Proposition 2.7},  $R_{*}:\pi_{i}\fGL{2n}\to \pi_{i}\fGL{4mn}$ is given by 
\[
R_{*}(y)=2m\esta(y) \quad \text{for} \quad y\in \pi_{i}\fGL{2n}.
\]

Moreover, by \cite{TAA2021}*{Lemma 2.9}, $R,L:\fGL{2m}\to \fGL{4mn}$ are pointed homotopic . Hence, $L_{*}=R_{*}:\pi_{i}\fGL{2m}\to \pi_{i}\fGL{4mn}$, which implies that 
\[
L_{*}(x)=2n\esta(x) \quad \text{for} \quad x\in \pi_{i}\fGL{2m}  \quad \text{and} \quad i\in \NN.
\]

Therefore, by the commutativity of squares \eqref{cd:cL} and \eqref{cd:cR}, we have that $L_{*}:\pi_{i}\fSp{m}\to \pi_{i}(\fO{4mn})$ and $R_{*}:\pi_{i}\fSp{n}\to \pi_{i}\fO{4mn}$ satisfy the following equalities
\[
(c''_{*}\circ L_{*})(x)=2n\esta(c'_{*}(x)) \quad \text{and} \quad (c''_{*}\circ R_{*})(y)=2m\esta(c'_{*}(y)).
\]
Thus, $(c''_{*}\circ L_{*})(x)=2nc'_{*}(x)$ and $(c''_{*}\circ R_{*})(y)=2mc'_{*}(y)$ for $x \in \pi_{i}\fSp{m}$, $y\in \pi_{i}\fSp{n}$ and $i<4m+2$. The result follows by equation \eqref{mixedproduct}.
\end{proof}

\begin{theorem} \label{T:cprima} (\cite[Theorem IV.5.16]{MiToToG1991})
Let $i$ be in the stable range for $\fSp{n}$. Then the induced homomorphism $c'_{*}:\pi_{i}\fSp{n} \to \pi_{i}\fGL{n}$
is an isomorphism if $i\equiv 3$ (mod 8) and is an isomorphism onto $2\pi_{i}\fGL{2n}$ if $i\equiv 7$ (mod 8). In all other cases, it is trivial because the source is torsion and the target is torsion-free.
\end{theorem}

\begin{corollary}\label{C:tpSpSp}
Let $m\leq n$ and $i<4m+2$, the homomorphism 
\begin{equation*}
\begin{tikzcd}
\stensor_{*}:\pi_{i}\fSp{m}\times\pi_{i}\fSp{n}  \arrow[r] & \pi_{i}\fO{4mn}
\end{tikzcd}
\end{equation*}
is given by
\[
\stensor_{*}(x,y)=%
\begin{cases}
nx+my & \text{if $i\equiv 3$ (mod 8)},\\
4(nx+my) & \text{if $i\equiv 7$ (mod 8)},\\
0 & \text{otherwise},\\
\end{cases}
\]
for $x \in \pi_{i}\fSp{m}$ and $y \in \pi_{i}\fSp{n}$.
\end{corollary}
\begin{proof}
Let $i<4m+2$ and $i\equiv 3,7$ (mod 8), then by Proposition \ref{P:tpSpSp}, diagrams \eqref{cd:cL} and \eqref{cd:cR} take the form

\begin{minipage}{0.48\linewidth}
\begin{equation*}
\begin{tikzcd}[row sep=large]
\pi_{i}\fSp{m} \arrow[r,"c'_{*}"] \arrow[d,"L_{*}"] &  \pi_{i}\fGL{2m} \arrow[d,"\times(2n)"]\\
\pi_{i}\fO{4mn} \arrow[r,"c''_{*}"] & \pi_{i}\fGL{4mn}
\end{tikzcd}
\end{equation*}
\end{minipage}
\begin{minipage}{0.48\linewidth}
\begin{equation*}
\begin{tikzcd}[row sep=large]
\pi_{i}\fSp{n} \arrow[r,"c'_{*}"] \arrow[d,"R_{*}"] &  \pi_{i}\fGL{2n} \arrow[d,"\times(2m)"]\\
\pi_{i}\fO{4mn} \arrow[r,"c''_{*}"] & \pi_{i}\fGL{4mn}.
\end{tikzcd}
\end{equation*}
\end{minipage}

\vspace{5mm}

By \cite{MiToToG1991}*{Theorem IV.5.12} and Theorem \ref{T:cprima},  if $i\equiv 3$ (mod 8), then $c'_{*}$ is an isomorphism and $c''_{*}$ is an isomorphism onto $2\pi_{i}\fGL{4mn}$. Therefore, $2L_{*}(x)=2nx$ and $2R_{*}(y)=2my$, i.e. $L_{*}(x)=nx$ and $R_{*}(x)=my$.

By \cite{MiToToG1991}*{Theorem IV.5.12} and Theorem \ref{T:cprima}, if $i\equiv 7$ (mod 8), then $c'_{*}$ is an isomorphism onto $2\pi_{i}\fGL{2l}$ for $l=m,n$ and $c''_{*}$ is an isomorphism. Therefore, $L_{*}(x)=4nx$ and $R_{*}(y)=4my$.

In all other cases, either the source or the target is trivial.
\end{proof}

\begin{corollary}\label{C:2tpSpSp}
Let $\stensor^{2}:\fSp{m} \to \fO{4m^{2}}$ be the map given by $\stensor^{2}(A) = A\stensor A$ for all $A \in \fSp{m}$. Then the homomorphism 
\begin{equation*}
\begin{tikzcd}
\stensor^{2}_{*}:\pi_{i}\fSp{m} \arrow[r] & \pi_{i}\fO{4m^{2}}
\end{tikzcd}
\end{equation*}
is given by
\[
\stensor^{2}_{*}(x)=%
\begin{cases}
2mx & \text{if $i\equiv 3$ (mod 8)},\\
8mx & \text{if $i\equiv 7$ (mod 8)},\\
0 & \text{otherwise},\\
\end{cases}
\]
for $x \in \pi_{i}\fSp{m}$ and all $i<4m+2$.
\end{corollary}

\subsection{Tensor product on the quotient in the stable range}
We describe the effect of the tensor product operation on the homotopy groups of the projective complex symplectic group. The methods we use in \cite{TAA2021} and \cite{TAAwOI2022} to establish the decomposition of topological Azumaya algebras apply to those whose degrees are relatively prime. For this reason, we only study the tensor product \eqref{qSp} when $n$ is odd.

The tensor product operation $\tensor:\fSp{m}\times\fSO{n} \to \fSp{mn}$ sends $Z(\fSp{m})\times\{I_{n}\}$ to $Z(\fSp{mn})$. As a consequence, the operation descends to the quotient 
\begin{equation}\label{qSp}
\begin{tikzcd}
\tensor:\fPSp{m}\times\fSO{n} \arrow[r] & \fPSp{mn}.
\end{tikzcd}
\end{equation}

\begin{proposition}\label{P:tppiiSp}
Let $i<4m+2$. The homomorphism 
\begin{equation*}
\begin{tikzcd}
\tensor_{*}:\pi_{i}\fPSp{m}\times\pi_{i}\fSO{n} \arrow[r] & \pi_{i}\fPSp{mn}
\end{tikzcd}
\end{equation*}
is given by $\tensor_{*}(x,y)=nx+2my$ for $x\in \pi_{i}\fPSp{m}$ and $y \in \pi_{i}\fSO{n}$.
\end{proposition}
\begin{proof}
If $i=0,1$, then $\tensor_{*}$ is trivial. Let $i>1$. There exits a map of fibrations
\begin{equation}\label{mfSp}
\begin{tikzcd}
Z(\fSp{m})\times\{I_{n}\} \arrow[r,hookrightarrow] \arrow[d,"\tensor"] & \fSp{m}\times\fSO{n} \arrow[r,twoheadrightarrow] \arrow[d,"\tensor"] & \fPSp{m}\times\fSO{n} \arrow[d,"\tensor"]\\
Z(\fSp{mn}) \arrow[r,hookrightarrow] & \fSp{mn} \arrow[r,twoheadrightarrow] & \fPSp{mn}.
\end{tikzcd}
\end{equation}
Then there exits a homomorphism between the long exact sequences associated to the fibrations in diagram \eqref{mfSp}. For $i>1$ we obtain a commutative square
\begin{equation*}
\begin{tikzcd}
\pi_{i}\fSp{m}\times\pi_{i}\fSO{n}\arrow[r,"\iso"] \arrow[d,"\otimes_{*}"] & \pi_{i}\fPSp{m}\times\pi_{i}\fSO{n} \arrow[d,"\otimes_{*}"]\\
\pi_{i}\fSp{mn} \arrow[r,"\iso"] & \pi_{i}\fPSp{mn}.
\end{tikzcd}
\end{equation*}
From this diagram and Proposition \ref{P:tpSpO} we have that for all $1<i<4m+2$, $\tensor_{*}(x,y)=nx+2my$ where $x\in \pi_{i}\fPSp{m}$ and $y \in \pi_{i}\fSO{n}$.
\end{proof}

To simplify notation, in the following sections we write $\Sp(n)$ instead of $\fSp{n}$ and similarly $\SO(n)$ instead of $\fSO{n}$, except in subsection titles.

\section{Decompostion of topological Azumaya algebras with symplectic involution} \label{sec:mainproof}

This section is mainly devoted to prove Theorem \ref{mainSp}. Additionally, we show that the map $f_{\tensor}: \B\PSp(m)\times\B\SO(n) \rightarrow \B\PSp(mn)$ defined below generally does not admit a section, and we provide examples of topological Azumaya algebras with symplectic involution for which no tensor product decomposition is possible when the dimension of the base space is above the stable range for the symplectic group.

Let $m$ and $n$ be positive integers such that $n$ is odd. By applying the classifying-space functor to the homomorphism \eqref{qSp} we obtain the following map
\begin{equation*}
\begin{tikzcd}
f_{\tensor}: \B\PSp(m)\times\B\SO(n) \arrow[r] & \B\PSp(mn).
\end{tikzcd}
\end{equation*}

Theorem \ref{mainSp} results as a corollary of Theorem \ref{T:BSp2mn}. The proof strategy for Theorem \ref{T:BSp2mn} relies on constructing a map $J: \B\PSp(m)\times \B\SO(n) \to \B\PSp(mn)\times \B\SO(N)$ which is $7$-connected and fits into the following homotopy-commutative diagram:

\begin{equation*}
\begin{tikzcd}[execute at begin picture={\useasboundingbox (-4.5,-1) rectangle (4.5,1);},row sep=large,column sep=huge]
& \B\PSp(m)\times \B\SO(n) \arrow[d,"f_{\tensor}"] \arrow[r,"J"] \arrow[d,"f_{\tensor}"] & \B\PSp(mn)\times \B\SO(N) \\
 X \arrow[r,"\cA"] \arrow[ur,dotted,"\cA_{2m}\times\cA_{n}",bend left=20]  & \B\PSp(mn). \arrow[ur,leftarrow,"\proj_{1}"] &
\end{tikzcd}
\end{equation*}

Here, $\proj_{1}$ denotes projection onto the first coordinate. Given the dimensional restriction $\dim(X)\leq 7$ and the connectivity of $J$, we apply the Whitehead theorem to conclude the existence of the required lift.

\subsection{A \texorpdfstring{$7$}{7}-connected map}
We construct the map $J$ here by first defining an intermediate map $\widetilde{T}$ in Proposition \ref{P:TtildaSp}, which is central to establishing the appropriate connectivity properties for $J$. We prove that the resulting map $J$ is $7$-connected.

\begin{proposition}\label{P:TtildaSp}
Let $m$ and $n$ be relatively prime positive integers such that $n$ is odd. Let $d$ denote $\min\{4m+2,n-1\}$. There exists a homomorphism $\widetilde{\Tr}:\PSp(m)\times \SO(n) \rightarrow \SO(N)$, for some positive integer $N$, such that for all $i<d$ the homomorphisms induced on homotopy groups 
\begin{equation*}
\begin{tikzcd}
\widetilde{\Tr}_{i}:\pi_{i}\PSp(m)\times \pi_{i}\SO(n) \arrow[r] & \pi_{i}\SO(N)
\end{tikzcd}
\end{equation*}
are given by the following expressions, where $u$ and $v$ are some positive integers, independent of $i$, for which $\bigl|vn-4um^{2}\bigr|=1$.
\begin{enumerate}
\item If $i=1$, then $\widetilde{\Tr}_{1}(x,y)=zx+y$ for $x,y \in \ZZ/2$ and some $z\in \ZZ/2$.
\item If $1<i<d$, then for all $x \in \pi_{i}\PSp(m)$ and $y \in \pi_{i}\SO(n)$ we have that
\[
\widetilde{\Tr}_{i}(x,y)=%
\begin{cases}
vy & \text{if $i\equiv 0, 1$ (mod 8)},\\
2umx+vy & \text{if $i\equiv 3$ (mod 8)},\\
8umx+vy & \text{if $i\equiv 7$ (mod 8)},\\
0 & \text{otherwise}.
\end{cases}
\]

\end{enumerate}
\end{proposition}
\begin{proof}
Without loss of generality suppose $4m+2<n-1$. Since $\gcd(m,n)=1$, there exist positive integers $u$ and $v$ satisfying the equation: $vn-4um^{2}=\pm1$. 

Set $N \coloneq 4um^{2}+vn$, and define $\Tr$ as the composite map:
\begin{equation*}
\begin{tikzcd}[column sep=large]
\Sp(m)\times \SO(n) \arrow[d,"(\stensor^{2}\text{,}\id)"]\\
\SO(4m^{2})\times \SO(n) \arrow[d,"(\oplus^{u}\text{,}\oplus^{v})"]\\
\SO(4um^{2})\times \SO(vn) \arrow[d,"\oplus"]\\
\SO(N).
\end{tikzcd}
\end{equation*}

Observe that the element $\bigl(\pm I_{m}, I_{n}\bigr)$ is sent to $\bigl(I_{4m^{2}},I_{n}\bigr)$ by $(\stensor^{2},\id)$, and subsequently to the identity element in $\SO(N)$ by the defined composite $\Tr$. Consequently, the map $\Tr$ factors through $\PSp(m)\times \SO(n)$ as follows:
\begin{equation*}
\begin{tikzcd}[column sep=huge]
\Sp(m)\times \SO(n) \arrow[rd,bend left=20,"\Tr"] \arrow[d,twoheadrightarrow]  &  \\
\PSp(m)\times \SO(n) & \SO(N). \arrow[l,leftarrow,"\widetilde{\Tr}"]
\end{tikzcd}
\end{equation*}

From Corollaries \ref{C:dsSp} and \ref{C:2tpSpSp} we have that for $i<4m+2$
\[
\Tr_{i}(x,y)=%
\begin{cases}
2umx+vy & \text{if $i\equiv 3$ (mod 8)},\\
8umx+vy & \text{if $i\equiv 7$ (mod 8)},\\
vy & \text{otherwise}.\\
\end{cases}
\]

The map of fibrations
\begin{equation*}
\begin{tikzcd}
Z(\Sp(m))\times \{I_{n}\} \arrow[r,hookrightarrow] \arrow[d] & \Sp(m)\times \SO(n) \arrow[r,twoheadrightarrow] \arrow[d,"\Tr"] & \PSp(m)\times \SO(n) \arrow[d,"\widetilde{\Tr}"]\\
\{I_{N}\} \arrow[r] & \SO(N) \arrow[r,equal] & \SO(N)
\end{tikzcd}
\end{equation*}
induces a commutative diagram
\begin{equation*}
\begin{tikzcd}
\pi_{i}\Sp(m)\times\pi_{i}\SO(n) \arrow[r,"\iso"] \arrow[d,"\Tr_{i}"] & \pi_{i}\PSp(m)\times\pi_{i}\SO(n) \arrow[d,"\widetilde{\Tr}_{i}"]\\
\pi_{i}\SO(N) \arrow[r,equal] & \pi_{i}\SO(N)
\end{tikzcd}
\end{equation*}
for $i>1$. Then $\widetilde{\Tr}_{i}(x,y)=\Tr_{i}(x,y)$ for $1<i<4m+2$.

For $i=1$, the map of fibrations induces the commutative diagram below
\begin{equation*}
\begin{tikzcd}
\pi_{1}\Sp(m)\times\pi_{1}\SO(n) \arrow[r,hookrightarrow] \arrow[d,"\Tr_{1}"] & \pi_{1}\PSp(m)\times\pi_{1}\SO(n) \arrow[d,"\widetilde{\Tr}_{1}"]\\
\pi_{1}\SO(N) \arrow[r,equal] & \pi_{1}\SO(N),
\end{tikzcd}
\end{equation*}
which takes the form
\begin{equation}\label{cd:TTtildeOSp}
\begin{tikzcd}
0\times \ZZ/2 \arrow[r,hookrightarrow] &  \ZZ/2 \times\ZZ/2  \arrow[d,"\widetilde{\Tr}_{1}"]\\
 & \ZZ/2. \arrow[ul,leftarrow,"\Tr_{1}"]
\end{tikzcd}
\end{equation}
Observe that the equality $vn-4um^{2}=\pm1$ implies $v$ is odd, then $\Tr_{1}(0,\beta)=\beta$, i.e. $\Tr_{1}$ is the identity. Thus, $\widetilde{\Tr}_{1}(0,1)=1$. Let $z$ denote $\widetilde{\Tr}_{1}(1,0)$. Then $\widetilde{\Tr}_{1}(x,y)=zx+y$.
\end{proof}

In what follows, $N$ is defined as in Proposition \ref{P:TtildaSp}. Let $J$ denote the map
\begin{equation}\label{JSpO}
\begin{tikzcd}[row sep=tiny]
J:\B\PSp(m)\times\B\SO(n) \arrow[r] & \B\PSp(mn)\times\B\SO(N)\\
(x,y) \arrow[r,mapsto] & \Bigl(f_{\tensor}(x,y)\text{,}\B\widetilde{\Tr}(x,y)\Bigr).
\end{tikzcd}
\end{equation}

Let $J_{i}$ denote the homomorphism induced by $J$ on homotopy groups:
\begin{equation}\label{JiOSp}
\begin{tikzcd}
J_{i}: \pi_{i}\B\PSp(m)\times \pi_{i}\B\SO(n) \arrow[r] & \pi_{i}\B\PSp(mn)\times\pi_{i}\B\SO(N)
\end{tikzcd}
\end{equation}

\begin{proposition}\label{P:JOSpi}
Let $m$ and $n$ be relatively prime positive integers such that $n$ is odd, and let $d$ denote $\min\{4m+3,n\}$. Then the homomorphism $J_{i}$ is an isomorphism if $0<i<d$ and $i\not \equiv 0$ (mod 8).
\end{proposition}
\begin{proof}
Without loss of generality we can assume $4m+3<n$. Let $0<i<4m+1$ and $i\not\equiv 0$ (mod 8). 

Observe that $J_{i}$ is trivial for $i=0,1$ and, $i\equiv 3,7$ (mod 8) with $i>2$. We then proceed by considering the cases $i=2$ and, $i\equiv 1,2,4,5,6$ (mod 8) with $i>2$.

\begin{enumerate}
\item Let $i=2$. By Propositions \ref{P:tppiiSp} and \ref{P:TtildaSp} the homomorphism $J_{2}:\ZZ/2\times\ZZ/2 \to \ZZ/2 \times\ZZ/2$ is represented by the invertible matrix
\[
\begin{pmatrix}
1 & 0\\
z & 1
\end{pmatrix}.
\]

\item Let $2<i<4m+1$.

\begin{itemize}
\item Let $i\equiv4$ (mod 8). By Propositions \ref{P:tppiiSp} and \ref{P:TtildaSp} the homomorphism $J_{i}:\ZZ\times \ZZ \to \ZZ \times\ZZ$ is represented by the invertible matrix
\[
\begin{pmatrix}
n & 2m\\
2um & v
\end{pmatrix}.
\]

\item Let $i\equiv1,2$ (mod 8). By Propositions \ref{P:tppiiSp} and \ref{P:TtildaSp} the homomorphism $J_{i}:0\times \ZZ/2 \to 0\times \ZZ/2$ is given by $J_{i}(0,y)=(2my,vy)=(0,y)$.

\item Let $i\equiv5,6$ (mod 8). By Propositions \ref{P:tppiiSp} and \ref{P:TtildaSp} the homomorphism $J_{i}:\ZZ/2\times 0 \to \ZZ/2\times 0$ is given by $J_{i}(x,0)=(nx,2umx)=(x,0)$.
\end{itemize}

\item Let $m$ be even, then $4m+1\equiv 1$ and $4m+2\equiv 2$ (mod 8). Thus, $J_{4m+1}$ and $J_{4m+2}$ take the form $0\times \ZZ/2 \to 0\times \ZZ/2$. As in the case $i\equiv 1,2$, we have $J_{4m+1}(0,y)=J_{4m+2}(0,y)=(0,y)$.

\item Let $m$ be odd, then $4m+1\equiv 5$ and $4m+2\equiv 6$ (mod 8). Thus, $J_{4m+1}$ and $J_{4m+2}$ take the form $\ZZ/2\times 0 \to \ZZ/2\times 0$. As in the case $i\equiv 5,6$, we have $J_{4m+1}(x,0)=J_{4m+2}(x,0)=(x,0)$.

\end{enumerate}
\end{proof}

As a corollary of Propositions \ref{P:JOSpi} we obtain the result below.

\begin{corollary}\label{C:J}
Let $m$ and $n$ be relatively prime positive integers such that $m>1$, $n>7$, and $n$ is odd. Then the map $J$ is $7$-connected.
\end{corollary}

\subsection{Lifting through \texorpdfstring{$f_{\tensor}:\B\fPSp{m}\times\B\fSO{n} \;\rightarrow\; \B\fPSp{mn}$}{f:BPSp2m x BSOn --> BPSp2mn}}


In this subsection, we proceed to prove Theorem \ref{mainSp}. We also observe that if the map $f_{\tensor}: \B\PSp(m)\times\B\SO(n) \to \B\PSp(mn)$ had a section, one could find a lift $\cA:X\to \B\PSp(mn)$ factoring through $f_{\tensor}$. Propositon \ref{nosecsymp} establishes that, in general, $f_{\tensor}$ does not admit such section. Finally, we provide examples of topological Azumaya algebras equipped with symplectic involutions defined over spaces whose dimension is outside the stable range for the simplectic group, for which no tensor product decomposition exists.

\begin{theorem}\label{T:BSp2mn}
Let $X$ be a CW complex such that $\dim(X)\leq 7$. Let $m$ and $n$ be relatively prime positive integers such that $m>1$, $n>7$, and $n$ is odd. Every map $\cA:X \rightarrow \B\PSp(mn)$ can be lifted to $\B\PSp(m)\times \B\SO(n)$ along the map $f_{\tensor}$.
\end{theorem}
\begin{proof}
Diagrammatically speaking, we want to find a map 
\begin{equation*}
\begin{tikzcd}
\cA_{2m}\times\cA_{n}:X \arrow[r] & \B\PSp(m)\times \B\SO(n)
\end{tikzcd}
\end{equation*}
such that diagram \eqref{cd:lpOSp} commutes up to homotopy
\begin{equation}\label{cd:lpOSp}
\begin{tikzcd}[execute at begin picture={\useasboundingbox (-4.5,-1) rectangle (4.5,1);},row sep=large,column sep=huge]
& \B\PSp(m)\times \B\SO(n) \arrow[d,"f_{\tensor}"] \\
 X \arrow[r,"\cA"] \arrow[ur,dotted,"\cA_{2m}\times\cA_{n}",bend left=20]  & \B\PSp(mn).
\end{tikzcd}
\end{equation}

Without loss of generality supposse $4m+2<n-1$. Corollary \ref{C:J} yields a map $J: \B\PSp(m)\times \B\SO(n) \rightarrow \B\SO(N)$ where $N$ is some positive integer so that $N\gg n-1>4m+2$. Observe that $f_{\tensor}$ factors through $\B\PSp(mn)\times \B\SO(N)$, so we can write $f_{\tensor}$ as the composite of $J$ and the projection $\proj_{1}$ shown in diagram \eqref{cd:JprojOSp}.

\begin{equation}\label{cd:JprojOSp}
\begin{tikzcd}[row sep=large,column sep=huge]
\B\PSp(m)\times \B\SO(n) \arrow[r,"J"] \arrow[d,"f_{\tensor}"] & \B\PSp(mn)\times \B\SO(N) \\
\B\PSp(mn) \arrow[ur,leftarrow,"\proj_{1}"]& 
\end{tikzcd}
\end{equation}

By the Whitehead theorem \cite[Corollary 7.6.23]{SpaAT2012}, since $J$ is $7$-connected and $\dim(X)\leq 7$, the map
\begin{equation*}
\begin{tikzcd}
J_{\#}:\bigl[X, \B\PSp(m)\times \B\SO(n)\bigr] \arrow[r] &  \bigl[X,\B\PSp(mn)\times \B\SO(N)\bigr]
\end{tikzcd}
\end{equation*}
is a surjection.

Let $s$ denote a section of $\proj_{1}$. The surjectivity of $J_{\#}$ implies $s\circ \cA$ has a preimage $\cA_{2m}\times\cA_{n}:X \rightarrow \B\PSp(2)\times \B\SO(n)$ such that $J\circ (\cA_{2m}\times\cA_{n})\simeq s\circ \cA$.

Commutativity of diagram \eqref{cd:lpOSp} follows from commutativity of diagram \eqref{cd:JprojOSp}. Thus, the result follows.
\end{proof}

\newpage
\begin{proposition}\label{nosecsymp}
Let $m$ and $n$ be positive integers such that $n$ is odd.
\begin{enumerate}
\item If $4m+4<n$ and $4m+4<4mn$, then $f_{\tensor}: \B\PSp(m)\times\B\SO(n) \to \B\PSp(mn)$ does not have a section.

\item If $n\in \{3,5,7\}$, $n<4m+3$ and $n<4mn$, then then $f_{\tensor}: \B\PSp(m)\times\B\SO(n) \to \B\PSp(mn)$ does not have a section.
\end{enumerate}
\end{proposition}
\begin{proof}
Suppose $f_{\tensor}$ has a section, namely $\sigma$. By Proposition \ref{tensorSpO}, the homomorphism induced on homotopy groups by $f_{\tensor}$ is given by $(x,y)\mapsto n\esta(x)+m\esta(d_{*}(y))$ for all $x\in \pi_{i}\B\PSp(m)$, $y\in\pi_{i}\B\SO(n)$, and $i \in \NN$.

Since $\sigma$ is a section of $f_{\tensor}$, the image of $(f_{\tensor})_{i}\circ \sigma_{i}$ is equal to $\pi_{i}\B\PSp(mn)$ for all $i\geq 1$.

\begin{enumerate}
\item From Table \ref{tab:lowhgSp} and \cite[Table 3]{TAAwOI2022}, $\pi_{4m+4}\B\SO(n)\iso\pi_{4m+4}\B\PSp(mn)\iso \ZZ$. From \cite{MiToHoG1963}, $\pi_{4m+4}\B\PSp(m)\iso\ZZ/2$. Therefore, the image of 
\[
(f_{\tensor})_{4m+4}:\pi_{4m+4}\B\PSp(m)\times\pi_{4m+4}\B\SO(n)\to\pi_{4m+4}\B\PSp(mn)
\]
is $m\ZZ$ which implies that the image of $(f_{\tensor})_{4m+4} \,\circ\, \sigma_{4m+4}$ is contained in $m\ZZ$. However, this image is equal to the integers.

\item Let $C$ denote the set $\bigl\{(8,3), (12,5), (16,7)\bigr\}$. From Table \ref{tab:lowhgSp}, $\pi_{i}\B\PSp(m)\iso\pi_{i}\B\PSp(mn)\iso \ZZ$  for all $(i,m) \in C$, and from \cite[Table 6.VII, Appendix A]{edm93} $\pi_{i}\B\SO(n)$ is torsion for all $(i,m) \in C$. Therefore, the image of $(f_{\tensor})_{i}:\pi_{i}\B\PO(m)\times\pi_{i}\B\SO(n)\to\pi_{i}\B\PO(mn)$ is $n\ZZ$ for all $(i,m) \in C$. In a similar way as in part (1), this leads to a contradiction.
\end{enumerate} 
\end{proof}

Let $m$ and $n$ be positive integers such that $n$ is odd. We present two examples of topological Azumaya algebras of degree $2mn$ with symplectic involutions that not decompose as the tensor product of topological Azumaya algebras of degrees $2m$ and $n$ with symplectic and orthogonal involutions, respectively.

\begin{example}\label{TAAwSIdontdecomp1}
Let $4m+4<n$ and $4m+4<4mn$. Let $\cS \in \pi_{4m+4}\B\PSp(mn)$ be a generator. Observe that  $\cS$ is a topological Azumaya algebra of degree $2mn$ with a symplectic involution on the unit $(4m+4)-$sphere, $S^{4m+4}$. An argument similar to the one in Proposition \ref{nosecsymp} part (1) can be used to prove that the tensor product decomposition mentioned above is not possible.
\end{example}

\begin{example}\label{TAAwSIdontdecomp2}
If $n\in \{3,5,7\}$, and $16<4m+3<4mn$. We can argue similarly for a topological Azumaya algebra of degree $2mn$ with a symplectic involution on the unit $i-$sphere that generates $\pi_{i}\B\PSp(mn)$ for $(i,m) \in C$.
\end{example}

\section{Decomposition of Symplectic Vector Bundles in the Stable Range} \label{sec:symplecticvb}
In this section, we show that symplectic vector bundles decompose as a tensor product of symplectic and orthogonal vector bundles in the stable range for $\SO(n)$.

\begin{proof}[Proof of Theorem \ref{vbSp}]
The tensor product operation $\tensor:\Sp(m) \times \SO(n) \to \Sp(mn)$ induces a map at the level of classifying spaces $f'_{\tensor}:\B\Sp(m) \times \B\SO(n) \to \B\Sp(mn)$. Let $F$ denote the homotopy fiber of $f'_{\tensor}$. From Proposition \ref{P:tpSpO}, we observe that $\pi_{i}F\iso \pi_{i}\B\SO(n)$ for all $i<4m+3$ and $i<n$. 

Therefore, the Moore-Postnikov tower of $f'_{\tensor}$ takes the form shown in diagram \eqref{MPofBSp} where $1<i<n-1$ and $i\equiv 3$ (mod $8$).
\begin{equation}\label{MPofBSp}
\begin{tikzcd}[row sep=small]
\B\Sp(m)\times\B\SO(n) \arrow[d] &\\
Y[n-1] \arrow[d] \arrow[r,"k_{n-1}"] & \K\Bigl(\pi_{n}\B\SO(n),n+1\Bigr)\\
Y[n-2] \arrow[d] \arrow[r,"k_{n-1}"] & \K\Bigl(\pi_{n-1}\B\SO(n),n\Bigr)\\
\vdots \arrow[d] & \\
Y[i+6] \arrow[d]  \arrow[r,"k_{i+6}"] & \K(\ZZ/2,i+8)\\
Y[i+5] \arrow[d]  \arrow[r,"k_{i+5}"] & \K(\ZZ/2,i+7)\\
Y[i+4] \arrow[d]  \arrow[r,"k_{i+4}"] & \K(\ZZ,i+6) \\
Y[i] \arrow[d] \arrow[r,"k_{i}"] & \K(\ZZ,i+2) \\
\vdots \arrow[d] & \\
\B\Sp(mn) \arrow[r,"k_{1}"] & \K(\ZZ/2,3)
\end{tikzcd}
\end{equation}

The quaternionic Grassmannian $G_{m}(\HH^{\infty})$ is a model for the classifying space of $\Sp(mn)$ \cite[pg 31-34]{HatchVB}. Moreover, from \cite[Section 3.2]{Hatch2002} and the universal coefficients theorem for cohomology $\Hy^{*}(G_{m}(\HH^{\infty});A)\iso A[\alpha_{1},\dots,\alpha_{m}]$ with $|\alpha_{i}|=4i$ for all $i=1,\dots,m$ and $A = \ZZ, \ZZ/2$ .

Let $\sk_{i}\B\Sp(mn)$ be the $i-$skeleton of $\B\Sp(mn)$. We have that $\Hy^{i}(\sk_{n}\B\Sp(mn)) \iso \Hy^{i}(\B\Sp(mn))$ for $i<n$, and $\Hy^{i}(\sk_{n}\B\Sp(mn)) \iso 0$ for $i>n$.

Since $\Hy^{i}(G_{m}(\HH^{\infty});A)\iso 0$ for $i\not\equiv 0$ (mod $4$), we have that $\Hy^{i}(\sk_{n}\B\Sp(mn);A)\iso 0$ for $i<n$,  $i\not\equiv 0$ (mod $4$) as well for $i\geq n+1$. Since $n$ is odd, it follows that $\Hy^{i}(\sk_{n}\B\Sp(mn);A)\iso 0$ for $i\not\equiv 0$ (mod $4$). Since $i+2, i+6, i+7$, and $i+8$ are not congruent to zero modulo $4$ for $i\equiv 3$ (mod $8$), there is a lifting, $\xi$, of the inclusion $\sk_{n}\B\Sp(mn) \hookrightarrow \B\Sp(mn)$ along the Moore-Postnikov tower of $f'_{\tensor}$.

Let $\cV: X\to \B\Sp(mn)$ be a symplectic bundle of rank $2mn$ and $\dim(X)\leq n$. By CW approximation, $\cV$ factors through the inclusion $\sk_{n}\B\Sp(mn) \hookrightarrow \B\Sp(mn)$. Thus, there exists a not necessarily unique lifting of $\cV$ along $f'_{\tensor}$ that makes diagram \eqref{liftingV} commute up to homotopy, namely the composite of the dotted arrow and $\xi$. The result follows.

\begin{equation}\label{liftingV}
\begin{tikzcd}[row sep=normal]
& & \B\Sp(m)\times\B\SO(n) \arrow[dd]\\
& \sk_{n}\B\Sp(mn)  \arrow[dr,hookrightarrow] \arrow[ur,"\xi",bend left=15] & \\
X \arrow[rr,"\cV"] \arrow[ru,dotted] &  & \B\Sp(mn)
\end{tikzcd}
\end{equation}
\end{proof}

\newpage
\bibliographystyle{alpha}
\bibliography{MyTAABiblio}
\end{document}